\documentclass[a4paper,12pt]{article}
\usepackage{amsmath, amsthm, amstext}
\usepackage{amssymb, array, amsfonts}
\usepackage[utf8]{inputenc}
\usepackage{graphicx}

\numberwithin{equation}{section}

\def \la{\lambda}

\def \ph{\varphi}
\def \oo{\omega}

\def \G{\Gamma}
\def \D{\Delta}

\def \O{\Omega}

\def \N{\mathbb{N}}
\def \R{\mathbb{R}}
\def \Z{\mathbb{Z}}

\def\n{\nabla}
\def\dd{\partial}

\def\1{1\!\!\!\!1}

\def\ran{\operatorname{Ran}}

\theoremstyle{plain}
\newtheorem{theorem}{\bf Theorem}[section]
\newtheorem{lemma}[theorem]{\bf Lemma}

\theoremstyle{definition}
\newtheorem{defi}[theorem]{\bf Definition}

\theoremstyle{remark}
\newtheorem{rem}[theorem]{\bf Remark}

\renewcommand{\le}{\leqslant}
\renewcommand{\ge}{\geqslant}

\renewcommand{\qed}{\vrule height7pt width5pt depth0pt}

\title{On the P\'olya conjecture for the Neumann problem in tiling sets}

\author{N.~D.~Filonov}
\date{}
\begin{document}
\maketitle

\begin{abstract}
In 1954, G.~P\'olya conjectured that the counting function of the eigenvalues 
of the Laplace operator of Dirichlet (resp. Neumann) boundary value problem
in a bounded set $\O\subset\R^d$ is lesser (resp. greater) than $C_W |\O| \la^{d/2}$.
Here $\la$ is the spectral parameter, and $C_W$ is the constant in the Weyl asymptotics.
In 1961, P\'olya proved this conjecture for tiling sets in the Dirichlet case,
and for tiling sets under some additional restrictions for the Neumann case.
We prove the P\'olya conjecture in the Neumann case for all tiling sets.
\end{abstract}

\section{Formulation of the result}
Let $\O \subset \R^d$ be a bounded open set.
Assume that the embedding $W_2^1 (\O) \subset L_2 (\O)$ is compact;
here $W_2^1 (\O)$ is the Sobolev space.
Let us consider the Laplace operator of the Dirichlet problem (resp. of the Neumann problem)
defined in the space $L_2 (\O)$ via the quadratic form
$$
\int_\O |\n\ph (x)|^2 dx, \qquad \ph \in \mathring W_2^1 (\O) \ \ (\text{resp.}\ \ph \in W_2^1 (\O)) .
$$
Here $\mathring W_2^1 (\O)$ is the closure of the set $C_0^\infty (\O)$ in the space $W_2^1 (\O)$.
We denote the eigenvalues of the Laplace operator of the Dirichlet problem by
\begin{equation*}
0 < \la_1 \le \la_2 \le \la_3 \le \dots, \qquad \la_k \to + \infty, 
\end{equation*}
and of the Neumann problem by
\begin{equation*}
0 = \mu_1 \le \mu_2 \le \mu_3 \le \dots, \qquad \mu_k \to + \infty,
\end{equation*}
taking multiplicity into account.
Note that if the set $\O$ is connected then $\la_2 > \la_1$ and $\mu_2 > 0$.
Note also that we need the compactness of the embedding $W_2^1 (\O) \subset L_2 (\O)$
for the case of the Neumann boundary conditions only. 
In the case of the Dirichlet boundary problem we need no such a condition,
because the embedding $\mathring W_2^1 (\O) \subset L_2 (\O)$ is compact for any bounded set $\O$.

The P\'olya conjecture says that
\begin{equation}
\label{1}
\la_k (\O) \ge \frac{4\pi^2}{\left(\oo_d |\O|\right)^{2/d}}\ k^{2/d} 
\end{equation}
and
\begin{equation}
\label{2}
\mu_{k+1} (\O) \le \frac{4\pi^2}{\left(\oo_d |\O|\right)^{2/d}}\ k^{2/d} 
\end{equation}
for all open bounded sets $\O$ and for all natural numbers $k$.
Here $\oo_d$ is the volume of the unit ball in $\R^d$,
and $|\O|$ is the measure of the set $\O$.
In terms of counting functions
$$
N_{\cal D} (\la, \O) := \# \{ k: \la_k < \la\}, \qquad
N_{\cal N} (\la, \O) := \# \{ k: \mu_k < \la\} ,
$$
the inequalities \eqref{1}, \eqref{2} are equivalent to 
\begin{equation}
\label{3}
N_{\cal D} (\la, \O) \le \frac{\oo_d |\O|}{(2\pi)^d}\ \la^{d/2},
\end{equation}
\begin{equation}
\label{4}
N_{\cal N} (\la, \O) \ge \frac{\oo_d |\O|}{(2\pi)^d}\ \la^{d/2}
\end{equation}
for all $\la > 0$.
Note that the formulation of the hypothesis by P\'olya himself in the book \cite{P1}
is slightly different, see comments in the next section.
The generally accepted name of the P\'olya conjecture is fixed for the inequalities \eqref{1}--\eqref{4}.

Later, P\'olya proved the estimate \eqref{1} for tiling sets.

\begin{defi}
\label{d1}
An open set $\O \subset \R^d$ is called tiling if the whole space $\R^d$ can be represented as a union of 
non-intersecting isometric copies $\O_j$ of the initial set $\O$ modulo a set of measure zero,
\begin{equation}
\label{5}
\left| \R^d \setminus \left(\bigcup_{j=1}^\infty\O_j\right)\right| = 0,
\qquad \O_j \cap \O_k = \emptyset \ \ \text{if} \  j \neq k .
\end{equation}
\end{defi}

\begin{rem}
\label{r2}
The shape of a tiling set can be rather complicated.
But in any case the measure of its boundary is zero.
Indeed, the condition \eqref{5} yields
$$
\dd \O_k \cap \O_j = \emptyset \qquad \text{for all} \ \ k, j \in \N,
$$
so
$$
\dd \O_1 \subset \R^d \setminus \bigcup_{j=1}^\infty\O_j.
$$
Thus, $|\dd\O| = 0$.
\end{rem}

\begin{theorem}[\cite{P2}]
\label{t3}
Let $\O$ be an open bounded tiling set in $\R^d$.
Then the inequalities \eqref{1} and \eqref{3} are fulfilled.
\end{theorem}

P\'olya obtained the result for the Neumann boundary problem under more restrictive conditions on the set $\O$.
Let us give some definitions.

\begin{defi}
A lattice in $\R^d$ is a set $\G$ of the following type
$$
\G = \left\{ v = \sum_{j=1}^d n_j v_j, \ \ n_j \in \Z\right\},
$$
where $\{v_1, \dots, v_d\}$ is a basis in $\R^d$.
\end{defi}

\begin{defi}
An open set $\O$ is called translationally tiling if there exists such a lattice $\G$ in $\R^d$ that
$$
\left(\O + v\right) \cap \O = \emptyset \qquad \text{for all} \ \ v \in \G \setminus \{0\},
$$
and
$$
\left| \R^d \setminus \bigcup_{v\in\G} \left(\O + v\right)\right| = 0.
$$
\end{defi}

\begin{defi}
\label{d6}
An open set $\O$ is called regularly tiling if there exist a natural number $l$ and sets
$$
\O_1, \dots, \O_l \quad \text{and} \quad \tilde \O,
$$
such that each set $\O_j$ is isometric to the set $\O$,
$$
\O_j \cap \O_k = \emptyset \ \ \text{if} \  j \neq k,
\qquad \bigcup_{j=1}^l \O_j\subset \tilde \O, 
\qquad \left| \tilde \O \setminus \left(\bigcup_{j=1}^l \O_j\right)\right| = 0,
$$
and the set $\tilde \O$ is translationally tiling.
\end{defi}

In other words, $\O$ is regularly tiling if the corresponding set of sets $\{\O_j\}_{j\in\N}$
in the Definition \ref{d1} is invariant under the shifts on vectors of some lattice.

\begin{theorem}[\cite{P2}]
\label{t7}
Let $\O$ be a polyhedron in $\R^d$.
Assume that $\O$ is a regularly tiling set.
Then the inequalities \eqref{2} and \eqref{4} are fulfilled.
\end{theorem}

The purpose of the present paper is to restore the equality of rights of Dirichlet and Neumann problems.
Let us formulate the result.

\begin{theorem}
\label{t8}
Let $\O$ be an open bounded tiling set in $\R^d$.
Assume that the embedding $W_2^1 (\O) \subset L_2 (\O)$ is compact.
Then the inequalities \eqref{2} and \eqref{4} are fulfilled.
\end{theorem}

\section{Remarks}

\vskip5mm
{\bf 2.1.}
There are different definitions of the counting function in the literature:
$$
N(\la) = \# \{ k: \la_k < \la\} \qquad \text{or} \qquad  N(\la) = \# \{ k: \la_k \le \la\}.
$$
The functions defined these ways differ from each other in the countable set of points
(the eigenvalues themselves).
For the hypotheses \eqref{3}, \eqref{4} the choice of definition has no importance 
because of continuity of the function in the right hand sides.

\vskip5mm
{\bf 2.2.}
It is easy to see that the formulas \eqref{1}--\eqref{4} are fulfilled if $d=1$.
Indeed, if $\O$ is an interval of the real axis, 
then the inequalities \eqref{1}, \eqref{2} become the equalities
$$
\la_k = \mu_{k+1} = \pi^2 k^2 |\O|^{-2} .
$$
Therefore, the estimates \eqref{3} and \eqref{4} are also true.
If $\O$ is a union of a finite or countable set of non-intersecting intervals,
$\O = \cup_j U_j$, then
$$
N_{\cal D} (\la, \O) = \sum_j  N_{\cal D} (\la, U_j) \le \sum_j |U_j| \pi^{-1} \sqrt\la = |\O| \pi^{-1} \sqrt\la,
$$
$$
N_{\cal N} (\la, \O) = \sum_j  N_{\cal N} (\la, U_j) \ge \sum_j |U_j| \pi^{-1} \sqrt\la = |\O| \pi^{-1} \sqrt\la.
$$
Thus, the estimates \eqref{1}--\eqref{4} hold.

\vskip5mm
{\bf 2.3.}
The initial formulation of the conjecture in the book \cite{P1} differs from the formulation given above.
First, P\'olya has formulated his conjecture and later proved in \cite{P2} Theorem \ref{t3} and Theorem \ref{t7}
in the 2D case only.
It is more natural to formulate the conjecture in arbitrary dimension $d$.
Moreover, P\'olya's proofs of Theorem \ref{t3} and Theorem \ref{t7} can also be repeated literally in any dimension.

Second, P\'olya conjectured the inequalities
$$
\la_k > \frac{4\pi k}{|\O|}, \qquad \mu_k < \frac{4\pi k}{|\O|}
$$
(for $d=2$). 
There are two distinctions here:
\begin{itemize}
\item the inequalities are strict, and this is stronger than \eqref{1}, \eqref{2};
\item the second inequality is for $\mu_k$, and not for $\mu_{k+1}$, and this assumption is weaker.
\end{itemize}
It seems to be more natural to formulate the conjectures as in \eqref{1}, \eqref{2}.
Note also that in \cite{P2} P\'olya proved just the inequalities
$$
\la_k \ge \frac{4\pi k}{|\O|} \ge \mu_{k+1} .
$$

\vskip5mm
{\bf 2.4.}
Theorem \ref{t3} and Theorem \ref{t8} imply the inequality
$\mu_{k+1} (\O) \le \la_k (\O)$ for all tiling sets.
Nothing is surprising here, this is Friedlander's inequality.
It is known for all bounded open sets $\O$ such that the embedding $W_2^1 (\O) \subset L_2 (\O)$ is compact,
see \cite{Fr}, \cite{Fil}.
If $d\ge 2$, then the inequality is strict, $\mu_{k+1} (\O) < \la_k (\O)$.

\vskip5mm
{\bf 2.5.}
In \cite{P2} P\'olya claims Theorem \ref{t7} for all regularly tiling sets. 
But he proved it for the regularly tiling polygons only, having said 
"the general case follows from this particular case by continuity".
It is not clear what did he mean by "by continuity";
it seems to be a mistake.

\vskip5mm
{\bf 2.6.}
The constants in the inequalities \eqref{1}--\eqref{4} can not be improved
as they coincide with the constants in the Weyl asymptotics
$$
N_{\cal D,\, N} (\la, \O) \sim \frac{\oo_d |\O|}{(2\pi)^d}\ \la^{d/2}, \qquad \la \to +\infty,
$$
see for example \cite{BS}, \cite{RS}.

\vskip5mm
{\bf 2.7.}
The Faber-Krahn inequality (see \cite{Faber}, \cite{Krahn}) says that the minimum of the first eigenvalue $\la_1 (\O)$
of the Dirichlet problem among all bounded sets $\O$ of given volume is reached on the ball,
$$
\la_1 (\O) \ge \la_1 (B_R) \qquad \text{with} \ \ \ |B_R| = |\O| .
$$
Freitas showed that this inequality implies the inequality \eqref{1} for all bounded open sets for some first numbers $k \in \{1, 2, \dots, b(d)\}$.
The number $b(d)$ grows with the dimension $d$ roughly speaking as $(e/2)^d$, see \cite{Freitas}.

\vskip5mm
{\bf 2.8.}
The following inequalities are known to hold for all open sets $\O$ and for all $\la > 0$:
\begin{itemize}
\item
Li-Yau's inequality
$$
N_{\cal D} (\la, \O) \le \left(\frac{d+2}d\right)^{d/2} \frac{\oo_d |\O|}{(2\pi)^d}\ \la^{d/2},
$$
and
\item
Kr\" oger's inequality
$$
N_{\cal N} (\la, \O) \ge \frac2{d+2} \ \frac{\oo_d |\O|}{(2\pi)^d}\ \la^{d/2},
$$
see \cite{LY}, \cite{Kr}, \cite{Laptev}.
\end{itemize}
These estimates are similar to \eqref{3}, \eqref{4}, 
but the constants are worse.

\vskip5mm
{\bf 2.9.}
Let $\O = \O_1 \times \O_2$ where $\O_1$ is a tiling set in $\R^{d_1}$, $d_1 \ge 2$,
and $\O_2 \subset \R^{d_2}$ is a bounded open set.
Then the estimates \eqref{1}, \eqref{3} hold, see \cite{Laptev}.

\section{Tiling sets}

If tiling sets that are not regularly tiling do exist? 
In \cite{P2} P\'olya wrote: {\it "Is there a plane-covering
domain which is not regularly plane-covering? 
The answer to this question does not seem to be obvious".}
Let us discuss this issue before moving on.

In 1900, D.~Hilbert announced his famous problem list.
The second part of the 18th problem was: "Is there a tiling polyhedra in $\R^3$ which is not 
an interior of a fundamental domain of any isometry group?"\
(Literally {\it "Whether polyhedra also exist which do not appear as fundamental regions of
groups of motions, by means of which nevertheless by a suitable juxtaposition of congruent copies a 
complete filling up of all space is possible",} see \cite{Hilbert}.)
In 1928, Reinhardt gave the positive answer having constructed such a (non-convex) set in $\R^3$ 
(\cite{Reinhardt}).
Later some examples were constructed in $\R^2$:
in 1935 Heesch found a non-convex decagon \cite{Heesch}, and in 1968 Kershner found a convex pentagon \cite{Kershner}.
In all these examples the tiling set is not a fundamental domain of any isometry group;
nevertheless the tiling is periodic.
In other words, all these sets a regularly tiling sets, see Definition \ref{d6} above.

We are interested in covering of $\R^d$ by isometric copies of one set.
The question on the covering by isometric copies of several sets was also considered.
Let us call these sets "prototiles". 
In 1961, Wang supposed \cite{Wang} that if there is a covering of the plane 
with finite number of prototiles, then this covering is invariant under the shift by some vector.
In 1966, Berger disproved this conjecture having constructed an example of non-periodic tiling with 20426 prototiles.
In 1971, Robinson reduced the number of prototiles to 6. 
In 1974, Penrose constructed such an example with only two prototiles (Penrose tiling).
For a long time the question
{\it Is there such example with only one prototile?} was open.
This question was also known as "einstein problem" (not after the name Albert Einstein,
but after a german word "ein stein").

Finally in 2011, Socolar and Taylor solved this problem and constructed an example of a tiling set such that any
its tiling of the plane is not periodic \cite{SocTay}.
Thus, a tiling set that is not regularly tiling does exist.
Note that the Socolar-Taylor prototile is not connected set.
The question if a non-periodic tiling of the plane with one connected (or even convex) prototile exists remains open.
\section{Lemmas}

The facts collected in this section are known.
We provide the proofs for the sake of completeness.

\begin{lemma}
\label{l41}
Let $L > 0$, $\la > 0$.
Then
$$
N_{\cal N} \left(\la, (-L,L)^d\right) \ge \oo_d \pi^{-d} L^d \la^{d/2} .
$$
\end{lemma}

\begin{proof}
Eigenfunctions of the Laplace operator of the Neumann problem in a cube $(-L,L)^d$ are
$$
\psi_{\vec n} (x) = \prod_{j=1}^d \cos \frac{\pi n_j (x_j +L)}{2L}, \qquad \vec n \in \N_0^d .
$$
Here $\N_0$ is the set of non-negative integeres.
The corresponding eigenvalues are 
$$
\mu_{\vec n} = \frac{\pi^2 \vec n^2}{4 L^2} ,
$$
so, 
$$
N_{\cal N} \left(\la, (-L,L)^d\right) = 
\# \left\{\vec n \in \N_0^d : |\vec n| < 2 L \pi^{-1} \sqrt \la \right\} .
$$
We associate each point $\vec n$ in the right hand side with a cube $\vec n + [0,1]^d$.
Such cubes cover a $2^d$-th part of the ball of radius $2 L \pi^{-1} \sqrt \la$,
$$
\left\{ \vec x \in (\R_+)^d : |\vec x| < 2 L \pi^{-1} \sqrt \la \right\} \subset
\bigcup_{\vec n \in \N_0^d : |\vec n| < 2 L \pi^{-1} \sqrt \la} \left( \vec n + [0,1]^d \right) .
$$
Therefore,
$$
N_{\cal N} \left(\la, (-L,L)^d\right) \ge \left|\left\{ \vec x \in (\R_+)^d : |\vec x| < 2 L \pi^{-1} \sqrt \la \right\}\right|
 = \frac{\oo_d}{2^d} \left(\frac{2L\sqrt\la}\pi\right)^d .
\qquad \qedhere
$$
\end{proof}

\begin{lemma}
\label{l42}
Let $\O$ and $\O_j$, $j = 1, \dots, m$, be open sets in $\R^d$ such that
$$
\bigcup_{j=1}^m \O_j \subset \O, \quad \O_j \cap \O_k = \emptyset \ \ \text{if} \ \ j \neq k,
\qquad \sum_{j=1}^m |\O_j| = |\O| .
$$
Then 
$$
N_{\cal N} (\la, \O) \le \sum_{j=1}^m N_{\cal N} (\la, \O_j) 
$$
for all $\la > 0$.
\end{lemma}

\begin{proof}
Fix $\la > 0$. 
Introduce subspaces
$$
{\cal F} = \ran E_{-\D_{\cal N} (\O)} [0, \la), \qquad 
{\cal F}_j = \ran E_{-\D_{\cal N} (\O_j)} [0, \la).
$$
Here $E_{-\D_{\cal N} (\O)}$ is the spectral projector of the Laplace operator of the Neumann problem in a set $\O$.
Then 
$$
{\cal F} \subset W_2^1 (\O), \qquad \dim {\cal F} = N_{\cal N} (\la, \O),
$$
$$
\int_{\O} |\n v(x)|^2 dx < \la \int_\O |v(x)|^2 dx \quad \text{for all} \ \ v \in {\cal F} \setminus \{0\} ;
$$
$$
{\cal F}_j \subset W_2^1 (\O_j), \qquad \dim {\cal F}_j = N_{\cal N} (\la, \O_j),
$$
and if a function $v_j \in W_2^1 (\O_j)$ is orthogonal to the subspace ${\cal F}_j$ 
in the sense of $L_2 (\O_j)$, then
$$
\int_{\O_j} |\n v_j (x)|^2 dx \ge \la \int_{\O_j} |v_j (x)|^2 dx .
$$

If there is a number $j$ such that $\dim {\cal F}_j = +\infty$, then there are nothing to prove.
Let us assume that
$$
\dim {\cal F}_j < +\infty \qquad \text{for all} \ \ j = 1, \dots, m.
$$
Assume moreover that
\begin{equation}
\label{6}
\dim {\cal F} > \sum_{j=1}^m \dim {\cal F}_j .
\end{equation}
Then one can choose a function $v \in {\cal F}$, $v\not\equiv 0$, such that 
$$
v_j \perp {\cal F}_j \ \ \text{in} \ L_2 (\O_j) \quad \text{for all} \ j = 1, \dots, m ,
$$
where $v_j : = \left. v \right|_{\O_j}$.
Now we have
$$
\int_\O |\n v(x)|^2 dx = \sum_{j=1}^m \int_{\O_j} |\n v_j (x)|^2 dx \ge 
\la \sum_{j=1}^m \int_{\O_j} |v_j (x)|^2 dx = \la \int_\O |v(x)|^2 dx,
$$
a contradiction.
Thus, \eqref{6} is not valid, and
$$
N_{\cal N} (\la, \O) \le \sum_{j=1}^m N_{\cal N} (\la, \O_j) . \qquad \qedhere
$$
\end{proof}

\begin{lemma}
\label{l43}
Let $\O_1 \subset \O_2$ be open sets in $\R^d$.
Assume that there exists a bounded linear extension operator
$$
\Pi : W_2^1 (\O_1) \to W_2^1 (\O_2) .
$$
Then for all $\la > 0$
$$
N_{\cal N} (\la, \O_1) \le N_{\cal N} \left(\|\Pi\|^2 (\la+1), \O_2\right) .
$$
\end{lemma}

\begin{proof}
Introduce the subspace
$$
{\cal L} = \ran E_{-\D_{\cal N} (\O_1)} [0, \la).
$$ 
Then
$$
{\cal L} \subset W_2^1 (\O_1), \qquad \dim {\cal L} = N_{\cal N} (\la, \O_1),
$$
$$
\int_{\O_1} |\n v(x)|^2 dx < \la \int_{\O_1} |v(x)|^2 dx \quad \text{for all} \ \ v \in {\cal L} \setminus \{0\} .
$$
Set ${\cal F } = \Pi {\cal L}$.
It is clear that
$$
{\cal F} \subset W_2^1 (\O_2) \quad \text{and} \quad \dim {\cal F} = \dim {\cal L} .
$$
Let $w \in {\cal F} \setminus \{0\}$.
Then $w = \Pi v$ for a function $v \in {\cal L} \setminus \{0\}$.
We have
\begin{eqnarray*}
\int_{\O_2} |\n w(x)|^2 dx \le \|w\|_{W_2^1(\O_2)}^2 \le \|\Pi\|^2 \|v\|_{W_2^1(\O_1)}^2 
= \|\Pi\|^2 \int_{\O_1} \left(|\n v(x)|^2 + |v(x)|^2\right) dx \\
< \|\Pi\|^2 (\la+1) \int_{\O_1} |v(x)|^2 dx 
\le \|\Pi\|^2 (\la+1) \int_{\O_2} |w(x)|^2 dx  .
\end{eqnarray*}
Therefore,
$$
N_{\cal N} \left(\|\Pi\|^2 (\la+1), \O_2\right) \ge \dim {\cal F} = \dim {\cal L} = N_{\cal N} (\la, \O_1) .
\quad \qedhere
$$
\end{proof}

\begin{lemma}
\label{l44}
Let $L > R > 0$.
Let $Q$ be a closed set,
$$
Q \subset (-L+R, L-R)^d.
$$
Put
$$
\O := (-L, L)^d \setminus Q .
$$
Then there exists a linear extension operator
$$
\Pi : W_2^1 (\O) \to W_2^1 \left((-L-R,L+R)^d \setminus Q\right) 
$$
such that $\|\Pi\| \le 2^{d/2}$.
\end{lemma}

\begin{proof}
Let $f \in W_2^1 (\O)$.
For every point $x \in (-L-R, L+R)^d \setminus (-L,L)^d$ we define a reflected point 
$\overline x \in [-L,L]^d \setminus [-L+R,L-R]^d$ as follows:
$$
\left(\overline x\right)_j = 
\begin{cases}
- 2 L - x_j, & \text{if}\ \  x_j \in (-L-R, -L), \\
x_j, & \text{if} \ \ x_j \in [-L, L], \\
2 L - x_j, & \text{if} \ \ x_j \in (L, L+R) ,
\end{cases}
$$
$j = 1, \dots, d$.
Each point $\overline x$ has at most $(2^d-1)$ preimages.
We define the extended function $\Pi f$ in $(-L-R, L+R)^d \setminus Q$ by the formula
$$
\Pi f(x) = \begin{cases}
f(x), & \text{if} \ \ x \in \O, \\
f(\overline x), & \text{if} \ \ x \in (-L-R, L+R)^d \setminus (-L,L)^d .
\end{cases}
$$
It is easy to see that
$$
\Pi f \in W_2^1 \left((-L-R, L+R)^d \setminus Q\right) 
$$
and
$$
\|\Pi f\|_{W_2^1 \left((-L-R, L+R)^d \setminus Q\right)}^2 \le 2^d \|f\|_{W_2^1(\O)}^2 . \quad \qedhere
$$ 
\end{proof}

\section{Proof of Theorem \ref{t8}}

Let $\O \subset \R^d$ be a bounded open tiling set.
Denote its diameter by $R$.
Fix a tiling of $\R^d$ by copies of $\O$,
\begin{equation*}
\left| \R^d \setminus \left(\bigcup_{j=1}^\infty\O_j\right)\right| = 0,
\qquad \O_j \cap \O_k = \emptyset \ \ \text{if} \  j \neq k ,
\end{equation*}
each set $\O_j$ is isometric to $\O$.
Let $L$ be a large parameter, $L > 2R$.

We introduce three sets of indices:
$$
I = \left\{ j \in \N : \O_j \subset (-L+R, L-R)^d \right\},
$$
$$
J = \left\{ j \in \N : \O_j \cap (-L, L)^d \neq \emptyset \right\},
\qquad K = J \setminus I .
$$
Obviously, $I \subset J$.
It is also clear that
$$
\bigcup_{j \in I} \O_j \subset (-L+R, L-R)^d ,
$$
and therefore,
\begin{equation}
\label{7}
\# I \le \frac{2^d (L-R)^d}{|\O|} < \frac{2^d L^d}{|\O|} .
\end{equation}
Next,
$$
\bigcup_{j \in J} \O_j \subset (-L-R, L+R)^d , \qquad
\bigcup_{j \in K} \O_j \subset (-L-R, L+R)^d \setminus (-L+2R, L-2R)^d ,
$$
so,
\begin{equation}
\label{8}
\# K \le \frac{2^d \left((L+R)^d - (L-2R)^d\right)}{|\O|} .
\end{equation}

Now, fix $\la > 0$.
Lemma \ref{l41} and Lemma \ref{l42} imply the inequality
$$
\oo_d \pi^{-d} L^d \la^{d/2} \le 
N_{\cal N} \left(\la, (-L,L)^d\right) 
\le \sum_{j\in I} N_{\cal N} (\la, \O_j)
+ N_{\cal N} \left(\la, (-L,L)^d \setminus \overline{\bigcup_{j\in I} \O_j}\right) .
$$
Here we have taken into account that $|\dd \O_j| = 0$ due to Remark \ref{r2}.
The estimate \eqref{7} yields
\begin{equation}
\label{9}
\oo_d \pi^{-d} L^d \la^{d/2} \le 
2^d L^d |\O|^{-1} N_{\cal N} (\la, \O) 
+ N_{\cal N} \left(\la, (-L,L)^d \setminus \overline{\bigcup_{j\in I} \O_j}\right) .
\end{equation}
By virtue of Lemma \ref{l44} there exists an extension operator
$$
\Pi : W_2^1 \left((-L,L)^d \setminus \overline{\bigcup_{j\in I} \O_j}\right)
\to W_2^1 \left((-L-R,L+R)^d \setminus \overline{\bigcup_{j\in I} \O_j}\right),
$$
and in particular, an extension operator
$$
\Pi_1 : W_2^1 \left((-L,L)^d \setminus \overline{\bigcup_{j\in I} \O_j}\right)
\to W_2^1 \left(\operatorname{int} \left(\overline{\bigcup_{j\in K} \O_j}\right)\right),
$$
with
$$
\|\Pi_1\| \le 2^{d/2} .
$$

Applying Lemma \ref{l43} and Lemma \ref{l42} once again, we get
\begin{eqnarray}
\nonumber
N_{\cal N} \left(\la, (-L,L)^d \setminus \overline{\bigcup_{j\in I} \O_j}\right)
\le N_{\cal N} \left(2^d (\la+1), \operatorname{int} \left(\overline{\bigcup_{j\in K} \O_j}\right)\right) \\
\le \sum_{j \in K} N_{\cal N} \left(2^d (\la+1), \O_j\right)
= \# K \cdot N_{\cal N} \left(2^d (\la+1), \O\right).
\label{10}
\end{eqnarray}
Now, \eqref{8}, \eqref{9} and \eqref{10} imply
$$
\oo_d \pi^{-d} L^d \la^{d/2} \le 
2^d L^d |\O|^{-1} N_{\cal N} (\la, \O) 
+ 2^d \left((L+R)^d - (L-2R)^d\right) |\O|^{-1} N_{\cal N} \left(2^d (\la+1), \O\right) .
$$
Therefore,
$$
N_{\cal N} (\la, \O) \ge 
\frac{\oo_d |\O| \la^{d/2}}{(2\pi)^d} 
- \frac{(L+R)^d - (L-2R)^d}{L^d} \ N_{\cal N} \left(2^d (\la+1), \O\right) .
$$
Taking the limit $L \to \infty$, we obtain \eqref{4}. \quad \qed


\end{document}